\documentclass{siamltex}
\usepackage{amssymb,amsmath,amsfonts,amscd}
\usepackage[ruled,section]{algorithm}
\usepackage{algorithmic,verbatim,url,comment,color}
\usepackage{hyperref,enumitem}

\newcommand{\Expect}[1]{\mbox{}{\bf{E}}\left[#1\right]}

\newcommand{\FNorm }[1]{\mbox{}\|#1\|_F  }
\newcommand{\FsNorm }[1]{\mbox{}\|#1\|_F  }
\newcommand{\FNormS}[1]{\mbox{}\|#1\|_F^2}

\newcommand{\TNorm}[1]{\mbox{}\|#1\|_2}
\newcommand{\TsNorm}[1]{\mbox{}\|#1\|_2}
\newcommand{\TNormS}[1]{\mbox{}\|#1\|_2^2}

\newcommand{\abs }[1]{\left|#1\right|}

\def\math#1{$#1$}

\def\frac#1#2{{#1\over #2}}

\newcommand{\remove}[1]{}

\newcommand{\bA}{\mathbf{A}}

\newcommand{\bV}{\mathbf{V}}
\newcommand{\bv}{\mathbf{v}}

\newcommand{\bU}{\mathbf{U}}
\newcommand{\bu}{\mathbf{u}}
\newcommand{\bQ}{\mathbf{Q}}
\newcommand{\bW}{\mathbf{W}}

\newcommand{\bX}{\mathbf{X}}

\newcommand{\bY}{\mathbf{Y}}
\newcommand{\by}{\mathbf{y}}
\newcommand{\bZ}{\mathbf{Z}}

\newcommand{\bI}{\mathbf{I}}
\newcommand{\bB}{\mathbf{B}}

\newcommand{\bK}{\mathbf{K}}

\newcommand{\bR}{\mathbf{R}}

\newcommand{\bT}{\mathbf{T}}

\newcommand{\bSigma}{\mathbf{\Sigma}}
\newcommand{\bPhi}{\mathbf{\Phi}}
\newcommand{\bPsi}{\mathbf{\Psi}}
\newcommand{\bTheta}{\mathbf{\Theta}}

\newcommand{\bzero}{\mathbf{0}}
\newcommand{\real}{\mathbb{R}}

\newcommand{\range}{\mathrm{range}}
\newcommand{\trace}{\mathrm{trace}}

\title{Structural Convergence Results for \\Approximation of Dominant Subspaces from\\ Block
Krylov Spaces\thanks{The work of the first and third authors was supported in part
by NSF grants IIS-1302231 and NSF IIS-1447283. The work of the second author was
supported in part by the XDATA Program of the Defense Advanced
Research Projects Agency (DARPA), administered through Air Force
Research Laboratory contract FA8750-12-C-0323.
The work of the fourth author was partially supported by NSF grant IIS 1124827, and by the Army Research Laboratory under Cooperative Agreement Number W911NF-09-2-0053 (the ARL Network Science CTA). The views and conclusions contained in this document are those of the authors and should not be interpreted as representing the official policies, either expressed or implied, of the Army Research Laboratory or the U.S. Government. The U.S. Government is authorized to reproduce and distribute reprints for Government purposes notwithstanding any copyright notation here on.}}

\author{Petros Drineas\thanks{Department of Computer Science,
Purdue University,
West Lafayette, IN, pdrineas@purdue.edu}
\and
Ilse C. F. Ipsen\thanks{Department of Mathematics, North Carolina State University, Raleigh, NC, ipsen@ncsu.edu}
\and
Eugenia-Maria Kontopoulou\thanks{Department of Computer Science,
Purdue University, West Lafayette, IN, ekontopo@purdue.edu}
\and
Malik Magdon-Ismail\thanks{Department of Computer Science,
Rensselaer Polytechnic Institute,
Troy, NY, magdon@cs.rpi.edu}
}

\begin{document}
\maketitle

\begin{abstract}
This paper is concerned with
approximating the dominant left singular vector space of a real matrix $\bA$ of arbitrary dimension,
from block Krylov spaces generated by the matrix $\bA\bA^T$ and the block vector $\bA\bX$.
Two classes of results are presented. First are bounds on the distance, in the two and Frobenius norms,
between the Krylov space and the target space. The distance is expressed in terms
of principal angles. Second are quality of approximation bounds, relative
to the best approximation in the Frobenius norm.
For starting guesses $\bX$ of full column-rank,
the bounds depend on the tangent of the principal angles between $\bX$ and the
dominant right singular vector space of $\bA$.
The results presented here form the structural foundation for
the analysis of randomized Krylov space methods. The innovative feature is a combination of
traditional Lanczos convergence analysis with optimal
approximations via least squares problems.
\end{abstract}


\begin{AM}
15A18, 15A42, 65F15, 65F20, 68W20.
\end{AM}

\section{Introduction}\label{s_intro}
Randomized methods for low-rank approximations from Kry\-lov spaces
are starting to emerge in the Theoretical Computer Science community \cite{MM2015,Wang2015}.
This motivated us to produce  a ``proof of concept'' for the approximation of dominant subspaces
from Krylov spaces.

\subsubsection*{Low-rank versus subspace approximations}
Our focus is the approximation of a dominant subspace of $\bA\in\real^{m\times n}$. This is a different
and harder problem than a low-rank approximation of~$\bA$.
To wit, the objective of a low-rank approximation is a matrix $\bZ$ with orthonormal
columns that makes $\|\bA-\bZ\bZ^T\bA\|$ small in some unitarily invariant norm \cite{HMT09_SIREV,TCS-060}.
In contrast, a subspace approximation aims at a space $\mathcal{K}$ that has a small angle with the
dominant target space, which is the space
spanned by the singular vectors associated with the top $k$ left singular vectors of $\bA$.

For a dominant subspace to be  well-defined, the
top $k$ singular values must be separated by a gap from the remaining singular values of $\bA$.
In contrast, a low-rank approximation can do without a singular value gap.
Accuracy results for dominant subspace computations are automatically informative for low-rank
approximations, but not vice versa. It is in this sense that dominant subspace approximations are harder.

\subsubsection*{This paper}
We consider block Krylov space methods for computing dominant left
singular vector spaces of general rectangular matrices and
we present structural, deterministic bounds  on
the quality of the subspaces, for essentially general starting guesses.
The innovative feature is a  fusion of eigenvalue and singular value technology:
We combine a traditional Lanczos convergence analysis~\cite{Saa11} with optimal
approximations via least squares problems \cite{BDM2011,BDM2014}.

Our long-term goal is to put randomized Krylov space approximations on a firm numerical footing.
However, at this  preliminary first step, we make a few idealized assumptions:
\begin{enumerate}
\item The block Krylov spaces have maximal dimension.
\item The analysis assumes exact arithmetic and does not address the implementation of
numerically stable recursions.
\end{enumerate}
Future work will need to deal with the challenging issues of finite precision arithmetic
and viable numerical implementations, including recursions,
numerical stability, maintaining orthogonality, deflation, adaptation of block size, and restarting.
Empirical evaluations will have to assess whether
the bounds are tight enough to be informative in practice.

\subsubsection*{Overview}
We start with a brief summary of our contributions (Section~\ref{s_results}), followed by
a comparison to existing work (Section~\ref{s_lit}). Auxiliary results (Section~\ref{s_aux})
set the stage for the proof of the main Theorems
(Sections \ref{sxn:prooft2}, \ref{s_prooft3}, \ref{s_prooft4}, and Appendix~\ref{s_prooft2a}).
We end the main part of the paper with a perspective on open problems  (Section~\ref{s_conc}).

\section{Results}\label{s_results}
After setting the context (Section~\ref{s_setting}), we give a brief summary of our bounds for:
The distance between the Krylov space and  the dominant left singular space (Section~\ref{s_angle});
a particular dominant subspace approximation from the Krylov space (Section~\ref{s_lowrank}); and
the polynomials appearing in the approximation (Section~\ref{sxn:randnlabounds}).
We end this section with a discussion of options
for bounding the distance between the initial guess and the dominant right singular
vector space (Section~\ref{s_x}).

\subsection{Setting}\label{s_setting}
To approximate the dominant left singular vector subspace
of a matrix $\bA \in \real^{m \times n}$, given a starting
guess $\bX\in\real^{n\times s}$, we construct
the Krylov space in\footnote{The superscript $T$ denotes the transpose, and $\|\cdot\|_2$ the two norm.}
$\bA\bA^T$ and $\bA\bX$,
\begin{equation}\label{eqn:KrylovSpace}
\mathcal{K}_q\equiv \mathcal{K}_q(\bA\bA^T,\bA\bX)=
\range\begin{pmatrix}\bA\bX  &(\bA\bA^T)\bA\bX & \cdots & (\bA\bA^T)^q\bA\bX\end{pmatrix}.
\end{equation}
We assume maximal dimension, $\dim(\mathcal{K}_q)= (q+1)s$.

In contrast to \cite{Baglama2005, Baglama2006}, the matrix $\bA$ occurs not only
in the powers $\bA\bA^T$
but also has a direct effect on the starting guess through $\bA\bX$. Furthermore, $\bX$ is
not required to have orthonormal columns and, at times, not even linearly independent columns.

Let $\bA=\bU\bSigma\bV^T$ be the \textit{full} SVD of $\bA$,
so that $\bSigma\in\real^{m\times n}$, and $\bU\in\real^{m\times m}$ and
$\bV\in\real^{n\times n}$ are orthogonal matrices.
For a positive integer $1\leq k<\rank(\bA)$, identify the dominant spaces by partitioning
$$\bSigma=\begin{pmatrix} \bSigma_k & \\ & \bSigma_{k,\perp}\end{pmatrix}, \qquad
\bU=\begin{pmatrix}\bU_k & \bU_{k,\perp}\end{pmatrix}, \qquad
\bV=\begin{pmatrix}\bV_k & \bV_{k,\perp}\end{pmatrix},$$
where the diagonal matrix $\bSigma_k$ contains the $k$ largest singular values,
hence is nonsingular. For the dominant  subspaces to be well-defined,
the dominant $k$ singular values of~$\bA$ must be strictly larger than the
remaining ones, $1/\|\bSigma_k^{-1}\|_2>\|\bSigma_{k,\perp}\|_2>0$.

\subsection{Krylov space angles}\label{s_angle}
We present
bounds for the distance between the Krylov space $\mathcal{K}_q$ and the
dominant left singular vector space $\range (\bU_k)$. Theorem~\ref{t_2} bounds the distance
between $\mathcal{K}_q$ and the whole space, while
Theorem~\ref{t_3} bounds the distance between $\mathcal{K}_q$ and
an individual left singular vector.
The distances are represented in terms of principal angles.

Theorem~\ref{t_2} below is in the spirit of Rayleigh-Ritz bounds \cite{BDR09,Dr96}.
It indicates how well the Krylov space $\mathcal{K}_q$ captures the targeted dominant
left singular vector space $\range(\bU_k)$ in both the two norm and the Frobenius norm.
Denote by $\bTheta(\mathcal{K}_q,\bU_k)\in \real^{k\times k}$ the diagonal matrix
of principal angles between $\mathcal{K}_q$ and $\range(\bU_k)$,
and by $\bTheta(\bX,\bV_k)\in\real^{k\times k}$ the diagonal matrix of the
principal angles between $\range(\bX)$ and $\range(\bV_k)$.  Principal angles are discussed
in detail  in Section~\ref{sxn:angles}.

\begin{theorem}\label{t_2}
Let $\phi(x)$ be a polynomial of degree~$2q+1$ with odd powers only, such that
$\phi(\bSigma_k)$ is nonsingular. If $\rank(\bV_k^T\bX)=k$, then\footnote{The subscript $F$ denotes the
Frobenius norm.}
\begin{eqnarray*}
\|\sin{\bTheta(\mathcal{K}_q,\bU_k)}\|_{2,F}&\leq& \|\phi(\bSigma_{k,\perp})\|_2\,\|\phi(\bSigma_k)^{-1}\|_2\>\|\bV_{k,\perp}^T\bX(\bV_{k}^T\bX)^{\dagger}\|_{2,F}.
\end{eqnarray*}
If, in addition, $\bX$ has orthornomal or linearly independent columns, then
$$\|\bV_{k,\perp}^T\bX(\bV_{k}^T\bX)^{\dagger}\|_{2,F}=\|\tan{\bTheta(\bX,\bV_k)}\|_{2,F}$$
and
\begin{eqnarray*}
\|\sin{\bTheta(\mathcal{K}_q,\bU_k)}\|_{2,F}&\leq& \|\phi(\bSigma_{k,\perp})\|_2\,
\|\phi(\bSigma_k)^{-1}\|_2\>\|\tan{\bTheta(\bX,\bV_k)}\|_{2,F}.
\end{eqnarray*}
\end{theorem}

\begin{proof} See Section~\ref{sxn:prooft2} for general and orthonormal $\bX$;
and Appendix~\ref{s_prooft2a} for $\bX$ with linearly independent columns.
\end{proof}
\medskip

Theorem~\ref{t_2} is reminiscent of the eigenvalue bounds  \cite[(2.18)]{Knya87}
which contain a tangent on the left.
The term $\|\bV_{k,\perp}^T\bX(\bV_{k}^T\bX)^{\dagger}\|_{2,F}$ already appeared
in previous analyses of randomized algorithms
\cite{Drineas2016,Drineas2008,Drineas2011,MahoneyD16}, and bounds for it are
discussed in Section~\ref{s_x}.
If $\bX$ is a random starting guess, such as
a random sign matrix, a random Gaussian matrix\footnote{The elements of a random Gaussian matrix
are independent identically distributed normal random variables with mean zero and
variance one.} or a matrix with randomly chosen orthonormal columns, then
 state-of-the-art matrix concentration inequalities can be called upon.

In the special case where $\bX$ has
linearly independent columns the bounds admit a geometric interpretation:
They depend on the tangents of angles
between $\range(\bX)$ and the dominant right singular vector space $\range(\bV_k)$.
The full-rank assumption for $\bV_k^T\bX$ means that the spaces
$\range(\bV_k)$ and $\range(\bX)$ are sufficiently close, with all principal angles
being less than $\pi/2$.

Next, Theorem~\ref{t_3} bounds the distances between $\mathcal{K}_q$ and
individual left singular vectors of $\bA$. To this end,
distinguish the $k$ dominant singular values and associated left singular vectors,
$$\bSigma_k=\diag\begin{pmatrix}\sigma_1& \cdots & \sigma_k\end{pmatrix},
\qquad \bU_k=\begin{pmatrix} \bu_1 & \cdots & \bu_k\end{pmatrix}.$$
\begin{theorem}\label{t_3}
Let $\phi(x)$ be a polynomial of degree~$2q+1$ with odd powers only, such that
$\phi(\bSigma_k)$ is nonsingular. If $\rank(\bV_k^T\bX)=k$, then
\begin{eqnarray*}
|\sin{\bTheta(\mathcal{K}_q,\bu_i)}|&\leq& \frac{\|\phi(\bSigma_{k,\perp})\|_2}{|\phi(\sigma_i)|}\>\|\bV_{k,\perp}^T\bX(\bV_{k}^T\bX)^{\dagger}\|_{2},
\qquad 1\leq i\leq k.
\end{eqnarray*}
If, in addition, $\bX$ has orthonormal columns, then
\begin{eqnarray*}
|\sin{\bTheta(\mathcal{K}_q,\bu_i)}|&\leq& \frac{\|\phi(\bSigma_{k,\perp})\|_2}{|\phi(\sigma_i)|}\>\|\tan{\bTheta(\bX,\bV_k)}\|_{2}, \qquad 1\leq i\leq k.
\end{eqnarray*}
\end{theorem}

\begin{proof}
See Section~\ref{s_prooft3}.
\end{proof}
\medskip

In the special case when $\bX$ has orthonormal columns,
the angle between a \textit{single} left singular vector and $\mathcal{K}_q$
is bounded by \textit{all} angles between $\bX$ and  the right singular
vector space $\range(\bV_k)$.

\subsection{Approximations from a Krylov space}\label{s_lowrank}
The results here are motivated by work in the Theoretical Computer Science community on
Randomized Linear Algebra \cite{Drineas2016}.
There, a common objective is the best rank-$k$ approximation to $\bA$
with respect to a unitarily invariant norm,
$$\bA_k \equiv \bU_k \bSigma_k \bV_k^T.$$
The particular approximation $\hat{\bU}_k$ computed
by Proto-Algorithm~\ref{alg:BL} guarantees a strong optimality property
in the projection $\hat{\bU}_k\hat{\bU}_k^T\bA$:
It is the best rank-$k$ approximation to $\bA$ from $\mathcal{K}_q$ with
respect to the Frobenius norm (see Lemma~\ref{lem:restate}).

\begin{algorithm}
\centerline{
\caption{Proto-algorithm for a low-rank approximation of $\bA$ from $\mathcal{K}_q$}\label{alg:BL}
}
\begin{algorithmic}[1]
\REQUIRE $\bA \in \real^{m \times n}$, starting guess $\bX \in \real^{n \times s}$\\
$\quad~$ Target rank $k<\rank(\bA)$, provided $\sigma_{k}>\sigma_{k+1}$\\
$\quad~$ Block dimension $q\geq 1$ with $k\leq (q+1)s\leq m$
\ENSURE $\hat{\bU}_k \in \real^{m \times k}$ with orthonormal columns
\STATE Set
$\bK_q=\begin{pmatrix}\bA\bX  &(\bA\bA^T)\bA\bX & \cdots & (\bA\bA^T)^q\bA\bX
\end{pmatrix}\in\real^{m\times (q+1)s}$,\\
and assume that $\rank(\bK_q)=(q+1)s$.
\STATE Compute an orthonormal basis $\bU_K\in\real^{m\times (q+1)s}$
for $\range(\bK_q)$.
\STATE Set $\bW\equiv \bU_K^T\bA\in\real^{(q+1)s\times n}$,
 and assume $\rank(\bW)\geq k$.
\STATE Compute an orthonormal basis $\bU_{W,k}\in\real^{(q+1)s\times k}$ for the $k$
dominant left singular vectors of $\bW$.
\STATE Return $\hat{\bU}_k = \bU_K \bU_{W,k}\in\real^{m\times k}$.
\end{algorithmic}
\end{algorithm}

Theorem~\ref{t_4} presents a quality-of-approximation result for $\hat{\bU}_k$.
To this end we distinguish the orthonormal columns of
$\hat{\bU}_k=\begin{pmatrix}\hat{\bu}_1 &\ldots& \hat{\bu}_k \end{pmatrix} \in\real^{m \times k}$
and set
\begin{eqnarray}\label{e_hatui}
\hat{\bU}_i\equiv \begin{pmatrix}\hat{\bu}_1 &\ldots& \hat{\bu}_i \end{pmatrix}
\in\real^{m\times i},\qquad 1\leq i\leq k,
\end{eqnarray}
and
$$\Delta \equiv \|\phi(\bSigma_{k,\perp})\|_2\>
\|\bV_{k,\perp}^T\bX\,(\bV_{k}^T\bX)^{\dagger}\|_{F}.$$

\begin{theorem}\label{t_4}
Let $\phi(x)$ be a polynomial of degree~$2q+1$ with odd powers only, such that
$\phi(\bSigma_k)$ is nonsingular,
and $\phi(\sigma_i)\geq \sigma_i$ for $1\leq i\leq k$.
If $\rank(\bV_k^T\bX)=k$, then for $1\leq i\leq k$,
\begin{eqnarray}
\FsNorm{\bA-{\hat\bU_i}{\hat\bU_i}^T\bA}&\leq&
\FsNorm{\bA-\bA_{i}}+\Delta\label{eqn:th1}\\
\TsNorm{\bA-{\hat\bU_i}{\hat\bU_i}^T\bA}&\leq&
\TsNorm{\bA-\bA_{i}}+\Delta\label{eqn:th2}\\
\sigma_i-\Delta &\leq& \TsNorm{\hat\bu_i^T\bA}\leq \sigma_i.\label{eqn:th3}
\end{eqnarray}
If, in addition, $\bX$ has orthonormal columns, then
$$\Delta=\|\phi(\bSigma_{k,\perp})\|_2\>\|\tan{\bTheta(\bX,\bV_k)}\|_{F}.$$
\end{theorem}

\begin{proof}
See Section~\ref{s_prooft4}.
\end{proof}
\medskip

Bounds of the form~(\ref{eqn:th3}) were already proposed in
\cite[Theorem 1]{MM2015} as a finer, \textit{vector-wise}, way to capture the quality of
approximations to individual left singular vectors of $\bA$.
Empirical evidence \cite{MM2015} suggests that error metrics of the form
(\ref{eqn:th1}) and~(\ref{eqn:th2}) indicate the quality of the
\textit{aggregate} approximation and are therefore coarser than~(\ref{eqn:th3}).

\subsection{Judicious choice of  polynomials}\label{sxn:randnlabounds}
We show the existence of and present bounds for the  polynomials
in Theorems \ref{t_2}, \ref{t_3}, and~\ref{t_4}. The strict inequality $\rank(\bA)>k$
in Algorithm~\ref{alg:BL} allows us to express the relative singular gap as
\begin{equation}\label{eqn:gamma}
\frac{\sigma_{k}-\sigma_{k+1}}{\sigma_{k+1}} \geq \gamma > 0,
\end{equation}
which is equivalent to $\sigma_k \geq (1+\gamma)\sigma_{k+1}>0$.

\begin{lemma}\label{lem:l1}
If (\ref{eqn:gamma}) holds, then there exists a polynomial $\phi(x)$ of degree $2q+1$
with odd powers only, such that
$\phi\left(\sigma_i\right) \geq \sigma_i>0$ for $1\leq i\leq k$, and
\begin{eqnarray*}
\abs{\phi\left(\sigma_i\right)}\leq \frac{4\sigma_{k+1}}{2^{(2q+1)\, \min
\left\{\sqrt{\gamma},1\right\}}}, \qquad i\geq k+1.
\end{eqnarray*}
Hence
$$\|\phi(\bSigma_k)^{-1}\|_2\leq \sigma_k^{-1} \qquad \mathrm{and}\qquad
\|\phi(\bSigma_{k,\perp})\|_2\leq
\frac{4\sigma_{k+1}}{2^{(2q+1)\, \min{\left\{\sqrt{\gamma},1\right\}}}}.$$
\end{lemma}

\begin{proof} See Section~\ref{s_l1proof}.
\end{proof}
\medskip

We apply Lemma~\ref{lem:l1} to the previous results, first for the special case when $\bX$ has linearly
independent columns. Abbreviate
$$\Gamma(\bTheta,\gamma,q) \equiv 4\,
\frac{\|\tan{\bTheta(\bX,\bV_k)}\|_{2}}{2^{(2q+1)\, \min\left\{\sqrt{\gamma},1\right\}}}.$$
To keep things short, we consider only the two-norm bound for Theorem~\ref{t_2}.

\begin{corollary}\label{c_23}
Let (\ref{eqn:gamma}) hold and $\rank(\bV_k^T\bX)=k$.
If $\bX$ has orthonormal columns, then
\begin{eqnarray*}
\|\sin{\bTheta(\mathcal{K}_q,\bU_k)}\|_{2}&\leq& \Gamma(\bTheta,\gamma,q)\,\frac{\sigma_{k+1}}{\sigma_k} \leq \frac{\Gamma(\bTheta,\gamma,q)}{1+\gamma},
\end{eqnarray*}
and
\begin{eqnarray*}
|\sin{\bTheta(\mathcal{K}_q,\bu_i)}|&\leq& \Gamma(\bTheta,\gamma,q)\,
\frac{\sigma_{k+1}}{\sigma_i},\qquad 1\leq i\leq k.
\end{eqnarray*}
\end{corollary}

\begin{proof}
Apply Lemma~\ref{lem:l1} to Theorems \ref{t_2} and~\ref{t_3}.
\end{proof}
\medskip

\begin{corollary}\label{c_4}
Let (\ref{eqn:gamma}) hold and $\rank(\bV_k^T\bX)=k$.
If $\bX$ has orthonormal columns, then Theorem~\ref{t_4} holds with
$$\Delta\leq \Gamma(\bTheta,\gamma,q)\, \sigma_{k+1},$$
so that for $1\leq i\leq k$
\begin{eqnarray*}
\FsNorm{\bA-{\hat\bU_i}{\hat\bU_i}^T\bA} &\le&
\FsNorm{\bA-\bA_{i}}+\Gamma(\bTheta,\gamma,q)\, \sigma_{k+1},\\
\TsNorm{\bA-{\hat\bU_i}{\hat\bU_i}^T\bA} &\le&
\TsNorm{\bA-\bA_{i}}+\Gamma(\bTheta,\gamma,q)\,\sigma_{k+1},\\
\sigma_i-\Gamma(\bTheta,\gamma,q)\, \sigma_{k+1}
 &\leq& \TsNorm{\hat\bu_i^T\bA}\le \sigma_i.
\end{eqnarray*}
\end{corollary}

\begin{proof}
Apply Lemma~\ref{lem:l1} to Theorem~\ref{t_4}.
\end{proof}
\medskip

To achieve an additive error of $\Gamma(\bTheta,\gamma,q) \leq \epsilon$,
set $q$ to be the smallest integer that exceeds
\begin{equation}\label{eqn:valq}
  q \geq \frac{1}{2\min\left\{\sqrt{\gamma},1\right\}}\left(\log_2 4\|\tan{\bTheta(\bX,\bV_k)}\|_{2}-\log_2{\epsilon}\right).
\end{equation}
Thus, as the singular value gap $\gamma$ decreases, the dimension of the space $\mathcal{K}_q$ increases.
More specifically,
$q$ increases logarithmically with higher target accuracy $\epsilon$ and increasing
distance of $\bX$ from the dominant right singular vector space of $\bA$.

If $\bX$ is rank deficient then Corollaries \ref{c_23} and~\ref{c_4} still hold with
$$ \Gamma(\bTheta,\gamma,q) = 4\,
\frac{\|\bV_{k,\perp}^T\bX(\bV_{k}^T\bX)^{\dagger}\|_{2}}{2^{(2q+1)\, \min
\left\{\sqrt{\gamma},1\right\}}}.$$

\subsection{The initial guess}\label{s_x}
It remains to bound $\|\bV_{k,\perp}^T\bX\,(\bV_{k}^T\bX)^{\dagger}\|_{2,F}$.
The simplest way might be strong submultiplicativity,
$$\|\bV_{k,\perp}^T\bX\,(\bV_{k}^T\bX)^{\dagger}\|_{2,F}\leq
\|\bV_{k,\perp}^T\bX\|_{2,F}\>\|(\bV_{k}^T\bX)^{\dagger}\|_2=
 \frac{\|\bV_{k,\perp}^T\bX\|_{2,F}}{\sigma_k(\bV_{k}^T \bX)},$$
followed by separate bounds for the individual factors.

Ideally, the starting guess $\bX$ should be close to
$\range(\bV_k)$ and far away  from $\range(\bV_{k,\perp})$, so that
$\sigma_k(\bV_k^T\bX)$ is large and
$\|\bV_{k,\perp}^T\bX\|_{2,F}$ is small.
The assumption $\sigma_k(\bV_{k}^T \bX)>0$  is critical for our results,
hence a necessary condition for the user-specified matrix $\bX\in\real^{n\times s}$
is $\rank(\bX)\geq k$, while
trying to keep the column dimension $s\geq k$ small.

If $\bX$ is a random Gaussian, then $\sigma_k(\bV_k^T\bX)$ is bounded
away from zero with high probability even for $s=k$.
However, there are many other choices for~$\bX$ that come with lower bounds for
$\sigma_k(\bV_k^T\bX)$. They include
random sign matrices \cite{Achlioptas2001,Magen2011},
the fast randomized Hadamard transform \cite{Ailon2009,Sarlos2006},
the subsampled randomized Hadamard transform \cite{Drineas2011,Tropp2010},
the fast randomized discrete cosine transform \cite{Nguyen2009}, and
input sparsity time embeddings \cite{Clarkson2013,Meng2013,Nelson2013}.

In contrast, keeping $\|\bV_{k,\perp}^T\bX\|_{F}$ small is relatively easy.
For typical random matrices $\bX$, one can show that, with high probability,
$$\|\bV_{k,\perp}^T\bX\|_{2,F}\leq c\,\|\bV_{k,\perp}\|_F\leq c\sqrt{n},$$
where $c$ is a small constant.

For instance, if $s=1$ and $\bX$ is a Gaussian column vector, then
\begin{equation}\label{eqn:bb}
\Expect{\FNormS{\bV_{k,\perp}^T\bX}}= \FNormS{\bV_{k,\perp}^T}\leq n.
\end{equation}
Markov's inequality guarantees that, with probability at least .9,
$$\FNorm{\bV_{k,\perp}^T\bX}\leq \sqrt{10n}.$$
Essentially all randomized embedding matrices
satisfy variants of (\ref{eqn:bb}), and we  expect
the iteration count $q$ in (\ref{eqn:valq}) to be logarithmic in~$n$.

From a numerical point of view, a starting guess $\bX$ with orthonormal
columns is preferable. Thus one could pick a random matrix $\bX$ and apply a
thin QR decomposition $\bX=\bQ\bR$.
However, this significantly complicates the derivation
of bounds for $\|\bV_{k,\perp}^T\bX\|_{2,F}$ and $\|(\bV_{k}^T\bX)^{\dagger}\|_2$,
as most matrix concentration inequalities apply
only to the original random matrix $\bX$, not to its orthonormal basis $\bQ$.
For instance, if $\bX$ is a random matrix whose entries
are $\pm 1$ with equal probability, then $\bQ$ does not inherit this property.
Fortunately,
the subsampled Hadamard transform \cite{Drineas2011,Tropp2010} is
one of a few random matrices with orthonormal columns, hence amenable to
application of matrix concentration inequalities.

\section{Comparison to existing work}\label{s_lit}
Our work on subspace computations is motivated by a recent probabilistic approach for
low-rank approximations via block Krylov spaces~\cite{MM2015}.

\paragraph{Randomized Methods}
Analyses of numerical methods that compute dominant subspaces and eigenvectors from
randomized starting vectors date back at least to the 1980s. They include
the power method and inverse iteration \cite{Dix83,JeI90}, and information theoretic analyses
of Lanczos methods \cite{KW92,KW94}.

Current analyses in Theoretical Computer Science focus on low-rank approximations \cite{HMT09_SIREV,TCS-060},
rather than subspace computations,  and as such tend not to produce bounds for the accuracy of subspaces
such as those in Section~\ref{s_angle}.

A popular approach towards low-rank approximation is subspace iteration, which makes use of
only the last iterate $(\bA\bA^T)^q\bA\bX$ \cite{HMT09_SIREV,TCS-060}.
Then came block Krylov methods, which exploit all of the iterates $(\bA\bA^T)^j\bA\bX$, $0\leq j\leq q$.
The analysis in~\cite{MM2015} relies on generalized matrix functions \cite{ABF2016,HBI1973},
but is limited to
Gaussian random matrices for starting guesses $\bX \in \real^{n \times k}$, and Chebyshev
polynomials for $\phi$.
The eponymous \textit{gap-dependent bound} \cite[Theorem 13]{MM2015}
requires a gap between the $k$th and $(k+1)$st singular values, and can be considered a
special case of Theorem~\ref{t_4}. However, \cite[Theorems 10, 11, and 12]{MM2015}
do not require a singular value gap such as (\ref{eqn:gamma}).
Such \textit{gap-independent bounds} are informative for low-rank approximations,
but not for computations of specific subspaces, as explained in Section~\ref{s_intro}.

Close on the heels of \cite{MM2015} is \cite{Wang2015}, with
a focus on gap-independent bounds of the type (\ref{eqn:th2})
and random Gaussian starting guesses \cite[Theorem 3.1]{Wang2015}.
The proof techniques in \cite{Wang2015} resemble ours, and leverage our prior work, see Lemma~\ref{lem:restate}
and \cite{BDM2011},  but numerical issues are not addressed.

\paragraph{Traditional, deterministic methods}
Although non-numerical in nature as well, our results are nevertheless
guided in spirit by foundational work on eigenvalue and invariant subspace
computations, including
the standard Lanczos convergence analysis~\cite[Section 6.6]{Saa11},
a geometric view of Krylov space methods~\cite{Beattie2004,Beattie2005},
block Lanczos methods~\cite{Li2015,Saa80}; and Rayleigh-Ritz bounds~\cite{BDR09,Dr96},
but also by Krylov space methods for singular value problems~\cite{Baglama2005, Baglama2006}.

A more detailed comparison, though, seems elusive due to differences in both,
the computational problem and the algorithm.
The analyses  in~\cite{Beattie2004,Beattie2005} target vector rather than block methods,
for eigenvalues and invariant subspaces of non-Hermitian matrices, with a concern
for restarting. The block methods in~\cite{Li2015,Saa80} are Lanczos methods for
Hermitian eigenvalue problems, and the analyses exploit the (block) tridiagonal
structure resulting from recursions. Although
singular value problems are considered in~\cite{Baglama2005}
and in~\cite{Baglama2006} with block methods, the Krylov spaces
are different and the focus is on algorithmic issues of augmenting and restarting
the Lanczos process, rather than subspace distances. Krylov spaces
For the solution of ill-posed least squares problems via LSQR, \cite{Jia2017} analyzes
the accuracy of a regularized solution, by bounding the sine between
$\mathcal{K}_j(\bA^T\bA, \bA^T\bv)$ and a dominant right singular vector space; however
all singular values must be distinct.

In the context of low-rank approximations,
\cite{SimonZha2000} proposed a Lanczos bidiagonalization with one-sided reorthogonalization.

In contrast, the context of this paper is singular vector spaces for general matrices
of any dimension; and an algorithm that is not tied to a particular recursion
and, due to steps 3-6, is not a straight-forward Krylov method. Furthermore,
the key feature of our analyses is a least squares approach \cite{BDM2011,BDM2014}
that assures the quality of the approximation.

\section{Auxiliary results}\label{s_aux}
We review submultiplicative inequalities for norms,
the matrix Pythagoras theorem, and solutions of multiple right-hand side
least  squares problems in the two norm (Section~\ref{sxn:la}). We also
present expressions for  elements of the Krylov space $\mathcal{K}_q$ (Section~\ref{s_kryprop}),
review angles between subspaces (Section~\ref{sxn:angles}),
and introduce gap-amplifying polynomials (Section~\ref{sxn:cheb}).

\subsection{Norm inequalities, Pythagoras, and least squares}\label{sxn:la}
We make frequent use of the strong sub-multiplicativity of the Frobenius norm
\cite[page 211]{HoJ91}. For matrices $\bY_1\in\real^{m\times k}$ and
$\bY_2\in\real^{k\times n}$,
\begin{eqnarray*}
\FNorm{\bY_1 \bY_2} &\leq& \TNorm{\bY_1}\FNorm{\bY_2}\\
\FNorm{\bY_1 \bY_2} &\leq& \FNorm{\bY_1}\TNorm{\bY_2}.
\end{eqnarray*}
If, in addition, $\rank(\bY_1)=\rank(\bY_2)$, then \cite[Theorem 2.2.3]{Bjo15}
\begin{equation}\label{eqn:pinv}
\left(\bY_1\bY_2\right)^{\dagger} = \bY_2^{\dagger}\bY_1^{\dagger}.
\end{equation}

\begin{lemma}[Matrix Pythagoras]\label{l_pyth}
Let $\bA,\bB\in\real^{m\times n}$. If $\bA^T\bB=\bzero$ then
$$\|\bA+\bB\|_F^2=\|\bA\|_F^2+\|\bB\|_F^2.$$
\end{lemma}

\begin{proof}
From $\bA^T\bB=\bzero$ and the linearity of the trace
\begin{eqnarray*}
\|\bA+\bB\|_F^2&=&\trace\left(\bA^T\bA+\bA^T\bB+\bB^T\bA+\bB^T\bB\right)\\
&=&\trace(\bA^T\bA)+\trace(\bB^T\bB)=\|\bA\|_F^2+\|\bB\|_F^2.
\end{eqnarray*}
\end{proof}

References to optimal solutions for multiple right-hand side least squares problems
in the two norm are hard to find.  Lemma~\ref{lem:2norm} below is easy to
prove in the Frobenius norm. An elaborate proof for Schatten $p$ norms
can be found in \cite[Theorems 3.2 and 3.3]{Mah92}. For the sake of completeness, we present a straightforward proof for the two norm.
\begin{lemma}\label{lem:2norm}
Let $\bA \in \real^{m \times n}$ and $\bB \in \real^{m \times p}$. Then\footnote{The superscript $\dagger$ denotes
the Moore-Penrose inverse.}
$\bA^{\dagger}\bB$ is a solution of
$$\min_{\bX \in \real^{n \times p}} \|\bB-\bA\bX\|_2$$
with least squares residual
$\|(\bI- \bA\bA^\dagger)\bB\|_2=
\min_{\bX \in \real^{n \times p}} \|\bB-\bA\bX\|_2$.
\end{lemma}
\medskip
\begin{proof}
Let $\bA = \bU \bSigma \bV^T$ be a thin SVD, and let
$\begin{pmatrix}\bU & \bU_{\perp}\end{pmatrix}\in\real^{m\times m}$
be an orthogonal matrix. Any $\bX\in\real^{n\times p}$ satisfies
\begin{eqnarray*}
  \TNormS{\bB-\bA\bX} &=& \TNormS{\bU\bU^T(\bB-\bA\bX)+\bU_\perp\bU_\perp^T(\bB-\bA\bX)} \\
  &=& \TNormS{\bU(\bU^T\bB-\bSigma\bV^T\bX)+\bU_\perp\bU_\perp^T\bB} \\
  &=& \TNormS{\bU\bT_1+\bU_\perp\bT_2},
\end{eqnarray*}
where $\bT_1\equiv\bU^T\bB-\bSigma\bV^T\bX$ and $\bT_2\equiv\bU_\perp^T\bB$.
 Let $\by_{opt}\in\real^p$ with $\|\by_{opt}\|_2=1$
satisfy $\TNormS{\bU_\perp\bT_2\by_{opt}} = \TNormS{\bU_\perp\bT_2}$.
The vector Pythagoras theorem implies
\begin{eqnarray*}
\TNormS{\bB-\bA\bX} &\geq &\TNormS{(\bU\bT_1+\bU_\perp\bT_2)\by_{opt}}\\
&=& \TNormS{\bU\bT_1\by_{opt}}+\TNormS{\bU_{\perp}\bT_2\by_{opt}}
\geq \TNormS{\bU_{\perp}\bT_2\by_{opt}}.
\end{eqnarray*}
Combining all of the above gives
$$\TNormS{\bB-\bA\bX}\geq\TNormS{(\bU\bT_1+\bU_\perp\bT_2)\by_{opt}} \geq \TNormS{\bU_{\perp}\bT_2}=\TNormS{\bU_{\perp}\bU_{\perp}^T\bB}.$$
This lower bound is achieved by $\bX_{opt}=\bA^{\dagger}\bB$,
$$\TNormS{\bB-\bA\bX_{opt}}=\TNormS{\bB-\bA\bA^{\dagger}\bB}=\TNormS{(\bI-\bU\bU^T)\bB}=\TNormS{\bU_{\perp}\bU_{\perp}^T\bB}.$$
\end{proof}

\subsection{The Krylov space}\label{s_kryprop}
The elements of the Krylov space $\mathcal{K}_q$ in (\ref{eqn:KrylovSpace})
can be expressed in terms of matrices
$\hat{\phi}(\bA\bA^T)\bA\bX\in\real^{m\times s}$, where $\hat{\phi}$ is a
polynomial of degree $q$. From the point of view of singular values, though, we
need a higher degree polynomial,
$$\hat{\phi}(\bA\bA^T)\,\bA\bX=\bU\,\hat{\phi}(\bSigma\bSigma^T)\bSigma\,\bV^T\bX=
\bU\,\phi(\bSigma)\,\bV^T\bX.$$
Here $\phi$ is a polynomial of degree $2q+1$ with odd powers only, and represents a
\textit{generalized matrix function} \cite{ABF2016,HBI1973}.
Since
$$\bSigma=\diag\begin{pmatrix}\sigma_1 & \cdots & \sigma_{\min\{m,n\}}\end{pmatrix}
\in\real^{m\times n}$$ is rectangular, the polynomial $\phi$ is applied to the diagonal elements
of $\bSigma$ only, and returns a diagonal matrix
of the same dimension,
$$\phi(\bSigma)\equiv\diag\begin{pmatrix}\phi(\sigma_1) & \cdots &
\phi(\sigma_{\min\{m,n\}})\end{pmatrix}\in \real^{m\times n}.$$
With this, we denote elements in $\mathcal{K}_q$ by
\begin{equation}\label{eqn:Phidef}
\bPhi\equiv \bU\phi(\bSigma)\bV^T\bX\> \in\real^{m\times s}.
\end{equation}
Clearly,
\begin{equation}\label{eqn:rangeK}
\range(\bPhi)\subset\mathcal{K}_q.
\end{equation}

The assumption $\dim(\mathcal{K}_q)=(q+1)s\leq m$ from Algorithm~\ref{alg:BL} guarantees
that $\bU_K$ is indeed an orthonormal basis for $\mathcal{K}_q$.

\subsection{Angles between subspaces}\label{sxn:angles}
Let $\bQ\in\real^{n\times s}$ and $\bW_k\in\real^{n\times k}$,
with $k\leq s$, be matrices with orthonormal columns. Hence, the singular values
$\sigma_j(\bW_k^T\bQ)$ lie between zero and one, and we can write
$\sigma_j(\bW_k^T\bQ)=\cos{\theta_j}$, $1\leq j\leq k$.
The \textit{principal} or \textit{canonical angles} between $\range(\bQ)$ and
$\range(\bW_k)$ are \cite[Section 6.4.3]{GovL13}
$$0\leq \theta_1\leq \cdots \leq\theta_k\leq \pi/2,$$
where $\theta_j = \cos^{-1}(\sigma_j(\bW_k^T\bQ))$. Following \cite[Definition I.5.3]{StS90}, we define the diagonal matrix of principal
angles between the subspaces spanned by the columns of $\bQ$ and the columns of $\bW_k$
$$\bTheta\left(\bQ,\bW_k\right)\equiv\diag\begin{pmatrix} \theta_1 & \cdots & \theta_k\end{pmatrix}.$$
Hence the singular values of $\bW_k^T\bQ$ are the diagonal elements of $\cos{\bTheta}\left(\bQ,\bW_k\right)$.
From \cite[Section 6.4.3]{GovL13} and \cite[Section I.5.3]{StS90} follows that
the distance between $\range(\bW_k)$ and $\range(\bQ)$ in the two and Frobenius norms, respectively, equals
\begin{eqnarray}\label{e_dist}
\|\sin{\bTheta}\left(\bQ,\bW_k\right)\|_{2,F}=\|(\bI-\bW_k\bW_k^T)\bQ\|_{2,F}.
\end{eqnarray}
In particular,  $\|\sin{\bTheta}\left(\bQ,\bW_k\right)\|_2=\sin{\theta_k}$, so the two norm distance
is determined by the largest principal angle.

Assume that $\range(\bW_k)$ and $\range (\bQ)$
are sufficiently close, so that the largest angle $\theta_k<\pi/2$.
This is equivalent to $\cos{\bTheta}\left(\bQ,\bW_k\right)$ being nonsingular, and $\rank(\bW_k^T\bQ)=k$.
Then \cite[Section 3]{ZKny13} implies that the tangents of the principal angles satisfy
\begin{eqnarray}
\nonumber \|\tan{\bTheta}\left(\bQ,\bW_k\right)\|_{2,F}&=&\|\sin{\bTheta\left(\bQ,\bW_k\right)}(\cos{\bTheta\left(\bQ,\bW_k\right)})^{\dagger}\|_{2,F}\\
\label{eqn:pdii1}&=& \left\|(\bI-\bW_k\bW_k^T)\bQ\,\left(\bW_k^T\bQ\right)^{\dagger}\right\|_{2,F}.
\end{eqnarray}
As above,  $\|\tan{\bTheta\left(\bQ,\bW_k\right)}\|_2=\tan{\theta_k}$, so the two norm tangent
is determined by the largest principal angle. The following lemma will be used in subsequent derivations,
so we include its simple proof.

\begin{lemma}[Theorem 3.1 in \cite{ZKny13}] \label{l_2}
Let $\bQ\in\real^{n\times s}$ have orthonormal columns, and let
$\bW\equiv\begin{pmatrix} \bW_k & \bW_{k,\perp}\end{pmatrix}\in\real^{n\times n}$
be an orthogonal matrix where $\bW_k\in\real^{n\times k}$ with $k\leq s$.
If $\rank(\bW_k^T\bQ)=k$ then
$$\|\tan{\bTheta(\bQ,\bW_k)}\|_{2,F} =\|(\bW_{k,\perp}^T\bQ)\,(\bW_k^T\bQ)^{\dagger}\|_{2,F}.$$
\end{lemma}
\medskip

\begin{proof}
From (\ref{eqn:pdii1}) and $\bI=\bW_k\bW_k^T+\bW_{k,\perp}\bW_{k,\perp}^T$  follows
\begin{eqnarray*}
\|\tan{\bTheta}\left(\bQ,\bW_k\right)\|_{2,F}&=& \left\|(\bI-\bW_k\bW_k^T)\bQ\,\left(\bW_k^T\bQ\right)^{\dagger}\right\|_{2,F}\\
&=& \left\|\bW_{k,\perp}\bW_{k,\perp}^T\bQ\,\left(\bW_k^T\bQ\right)^{\dagger}\right\|_{2,F}\\
&=& \left\|\bW_{k,\perp}^T\bQ\,\left(\bW_k^T\bQ\right)^{\dagger}\right\|_{2,F}.
\end{eqnarray*}
The last equality follows from the unitary invariance of the two and the Frobenius norms.
\end{proof}

\subsection{Gap-amplifying and Chebyshev polynomials}\label{sxn:cheb}
\label{section:gap-amplifying}
We generalize the Chebyshev-based
\textit{gap-amplifying polynomials} in \cite[Section 4.4]{MM2015}, \cite[Section 2.2]{Wang2015}.
Given an integer $q\geq 1$, define the polynomial
$$\psi_{q^{\prime}}(x) = \psi_{2q+1}(x)=\sum_{j=0}^q a_{2j+1}x^{2j+1}$$
of degree  $q^{\prime}=2q+1$ with only odd powers of $x$. The polynomial $\psi$
is \textit{gap-amplifying} if it satisfies three properties:
\begin{enumerate}
\item Small input values remain small,
$$\psi_{q'}(1)=1, \qquad \text{and} \qquad |\psi_{q'}(x)|\leq 1\quad \text{for}\ x\in[0,1].$$
\item Large input values are amplified,
$$\psi_{q'}(x)\geq x\, c\,2^{q' \, r(x)}\quad \text{for} \ x\geq 1,$$
where the constant $c$ and the function $r(x)$ are parameters of $\psi$.
\item Super linear growth for large input values,
$$\frac{\psi_{q'}(x)}{x}\geq\frac{\psi_{q'}(y)}{y}\quad \text{for}\  x\geq y\geq 1.$$
\end{enumerate}

The simplest gap-amplifying polynomial is a Chebyshev polynomial.
\begin{lemma}[Lemma 5 in \cite{MM2015}]\label{lemma:cheb}
The Chebyshev polynomial  $T_{q'}(x)$ of the first kind contains only odd powers of $x$ and is gap-amplifying with $c=1/4$ and $r(x)=\min\{\sqrt{x-1},1\}$.
\end{lemma}
\medskip

\begin{proof}
We give a quick sketch of the proof of \cite[Lemma 5]{MM2015}).
Clearly, the Chebyshev polynomial $T_{q'}(x)$ satisfies Property 1.
To prove Property 3, it suffices to show
\begin{eqnarray}\label{lemma:cheb:1}
T^{\prime}_{q^{\prime}}(z)\geq \frac{T_{q^{\prime}}(y)}{y} \quad \text{for some}\ z\geq y\geq 1,
\end{eqnarray}
because the mean value theorem implies there exists a $z\in[y,x]$ with
\begin{eqnarray*}
T_{q^{\prime}}(x)&=&T_{q^{\prime}}(y+x-y)=T_{q'}(y)+T^{\prime}_{q^{\prime}}(z)(x-y)\\
&\geq& T_{q^{\prime}}(y)+\frac{T_{q^{\prime}}(y)}{y}(x-y)=
x\frac{T_{q^{\prime}}(y)}{y}.
\end{eqnarray*}
Our proof of Property 2 corrects a small typo in a similar proof  in~\cite{MM2015}.
Although a bound equivalent to Property 2 is claimed in
\cite[Lemma 5]{MM2015} it is  only proved that $T_{q'}(x)\geq c\,2^{q'\, r(x)}$.
However, only a slight modification is required for the stronger result. The proof of
\cite[Lemma 5]{MM2015} shows that
\begin{eqnarray*}
T_{q^{\prime}}(x)\geq \tfrac{1}{2}\,2^{q^{\prime}\, \sqrt{x-1}} \quad \text{for}\ 1\leq x\leq 2.
\end{eqnarray*}
To derive a lower bound on $T_{q^{\prime}}(x)/x$ for $x\geq 1$,
first consider $1\leq x\leq 2$, where
$$\frac{T_{q'}(x)}{x}\geq \frac{T_{q^{\prime}}(x)}{2}\geq \tfrac{1}{4}\, 2^{q^{\prime} \sqrt{x-1}}.$$
For the remaining case $x>2$, Property 3 implies
$$\frac{T_{q^{\prime}}(x)}{x}\ge\frac{T_{q^{\prime}}(2)}{2}\geq
\tfrac{1}{4}\, 2^{q^{\prime}}.$$
Hence
$T_{q^{\prime}}(x)\geq  \frac{x}{4}\, 2^{q^{\prime}\,\min\{\sqrt{x-1},1\}}$.
\end{proof}
\medskip

For our analysis, we use a rescaled version of the
gap-amplifying polynomial $T_{q^{\prime}}(x)$, which has similar
properties to the original.

\begin{lemma}\label{lem:prop1}
Let
\begin{equation}\label{eqn:phi1}
\phi(x)=\frac{(1+\gamma)\,\alpha}{\psi_{q'}\left(1+\gamma\right)}\psi_{q'}(x/\alpha)
\end{equation}
be the rescaled gap-amplifying polynomial. Then
$$\abs{\phi(x)} \leq \frac{4\alpha}{2^{q'\,\min
\left\{\sqrt{\gamma},1\right\}}}\qquad \text{for}\ 0\leq x\leq \alpha.$$
\end{lemma}

\begin{proof}
The proof is immediate since
$\abs{\psi_{q^{\prime}}(x/\alpha)}\leq 1$ for  $0\leq x\leq \alpha$,
and Property 2 implies
$$\psi_{q^{\prime}}\left(1+\gamma\right) \geq
\frac{(1+\gamma)}{4}\, 2^{q^{\prime}\, \min\left\{\sqrt{\gamma},1\right\}}.$$
\end{proof}
\medskip

\begin{lemma}\label{lem:prop2}
The rescaled gap-amplifying polynomial $\phi(x)$ in (\ref{eqn:phi1}) satisfies
$$\phi(x)\geq x \qquad \text{for}~ x\geq (1+\gamma) \alpha.$$
\end{lemma}
\begin{proof}
Property 3 implies
$$\frac{\psi_{q^{\prime}}(x/\alpha)}{x/\alpha} \geq
\frac{\psi_{q^{\prime}}(1+\gamma)}{1+\gamma}
\qquad \text{for}~ x\geq (1+\gamma)\alpha.$$
Now rearrange terms and apply the definition of $\phi(x)$ in (\ref{eqn:phi1}).
\end{proof}

\subsection{Proof of Lemma~\ref{lem:l1}}\label{s_l1proof}
Let \math{\phi(x)} be the rescaled gap-amplifying polynomial in (\ref{eqn:phi1}) with
$\alpha=\sigma_{k+1}$ and $\gamma$ in (\ref{eqn:gamma}).
The inequalities
for $\phi(\sigma_i)$ follow from $q'=2q+1$, Lemma~\ref{lem:prop1} and
Lemma~\ref{lem:prop2}.

From $\phi(\sigma_i)>0$ for $1\leq i\leq k$ and $\phi(\bSigma_k)$ being
a diagonal matrix follows
$$\|\phi(\bSigma_k)^{-1}\|_2=\max_{1\leq i\leq k}{\phi(\sigma_i)^{-1}}
\leq \max_{1\leq i\leq k}{\sigma_i}^{-1}=\sigma_k^{-1}.$$
Furthermore
\begin{eqnarray*}
\|\phi(\bSigma_{k,\perp})\|_2=\max_{i\geq k+1}{|\phi\left(\sigma_i\right)|}\leq
\frac{4\sigma_{k+1}}{2^{(2q+1)\, \min{\left\{\sqrt{\gamma},1\right\}}}}.
\end{eqnarray*}

\section{Proof of Theorem~\ref{t_2}, for general and orthonormal $\bX$}\label{sxn:prooft2}
We focus on the case where $\bX$ is a general matrix, or has orthonormal columns, and
postpone the technicalities required for the full-column rank case to Appendix~\ref{s_prooft2a}.

The first and most critical step of the proof makes a connection between principal angles
and least-squares residuals.

\paragraph{Viewing the sine as a least squares residual}
Let $\mathcal{P}_q$ be the orthogonal projector onto the Krylov space $\mathcal{K}_q$.
For $\bPhi$ in (\ref{eqn:Phidef}) let
$\bPhi\bPhi^{\dagger}$ be the orthogonal projector onto $\range(\bPhi)$,
with $\range(\bPhi\bPhi^{\dagger})\subset\range(\mathcal{P}_q)$ due to
(\ref{eqn:rangeK}). Hence, (\ref{e_dist}) implies
\begin{eqnarray}\label{e_sin}
\|\sin{\bTheta(\mathcal{K}_q,\bU_k)}\|_{2,F}&=&\|(\bI-\mathcal{P}_q)\,\bU_k\|_{2,F}\leq
\|(\bI-\bPhi\bPhi^{\dagger})\,\bU_k\|_{2,F}.
%&=& \|\sin{\bTheta(\bPhi,\bU_k)}\|_{2,F}.
\end{eqnarray}
Lemma~\ref{lem:2norm} implies that
$\|(\bI-\bPhi\bPhi^{\dagger})\,\bU_k\|_{2,F}$ is the residual
of the least squares problem
$$\|(\bI-\bPhi\bPhi^{\dagger})\,\bU_k\|_{2,F}=\min_{\bPsi}{\|\bU_k-\bPhi\bPsi\|_{2,F}}
=\|\bU_k-\bPhi\bPsi_{opt}\|_{2,F},$$
where $\bPsi_{opt}=\bPhi^{\dagger}\bU_k$ is a least squares solution.

\paragraph{Focussing on the target space}
Decompose $\bPhi$ into the target component $\range(\bU_k)$ and the complementary subspace,
$\bPhi=\bPhi_k+\bPhi_{k,\perp}$, where
\begin{eqnarray}\label{e_phidecomp}
\bPhi_k\equiv\bU_k\phi(\bSigma_k)\bV_k^T\bX, \qquad
\bPhi_{k,\perp}\equiv \bU_{k,\perp}\phi(\bSigma_{k,\perp})\bV_{k,\perp}^T\bX.
\end{eqnarray}
From $\rank(\bV_k^T\bX)=k$ follows that $(\bV_k^T\bX)^{\dagger}$ is a right inverse,
$(\bV_k^T\bX) (\bV_k^T\bX)^{\dagger}=\bI_k$. With (\ref{eqn:pinv}) this gives
\begin{eqnarray}\label{e_proj}
\bPhi_k^{\dagger}= (\bV_{k}^T\bX)^\dagger \phi(\bSigma_k)^{-1}\bU_k^T
\quad \text{and} \quad
\bPhi_k\bPhi_k^{\dagger}=\bU_k\bU_k^T,
\end{eqnarray}
meaning $\bPhi_k\bPhi_k^{\dagger}$ is the orthogonal projector onto the target space
$\range(\bU_k)$.
The minimality of the least squares residual implies
\begin{eqnarray*}\label{e_ls}
\|(\bI-\bPhi\bPhi^{\dagger})\,\bU_k\|_{2,F}&=&\|\bU_k-\bPhi\,(\bPhi^{\dagger}\bU_k)\|_{2,F}\nonumber\\
&\leq&  \|\bU_k-\bPhi\,(\bPhi_k^{\dagger}\bU_k)\|_{2,F}
=\|(\bI-\bPhi\bPhi_k^{\dagger})\,\bU_k\|_{2,F}.
\end{eqnarray*}
Now replace the other instance of $\bPhi$ by (\ref{e_phidecomp}), and use (\ref{e_proj})
to simplify
\begin{eqnarray}\label{e_ls1}
\|(\bI-\bPhi\bPhi^{\dagger})\,\bU_k\|_{2,F}&\leq &\|(\bI-\bPhi\bPhi_k^{\dagger})\,\bU_k\|_{2,F}
=\|(\bI-\bPhi_k\bPhi_k^{\dagger})\bU_k-\bPhi_{k,\perp}\bPhi_k^{\dagger}\bU_k\|_{2,F}\nonumber\\
&=&\|(\bI-\bU_k\bU_k^{T})\bU_k-\bPhi_{k,\perp}\bPhi_k^{\dagger}\bU_k\|_{2,F}\nonumber\\
&=&\|\bPhi_{k,\perp}\bPhi_k^{\dagger}\bU_k\|_{2,F}.
\end{eqnarray}

\paragraph{Summary so far}
Combining  (\ref{e_sin}) with (\ref{e_ls1}) gives
\begin{eqnarray*}
\|\sin{\bTheta(\mathcal{K}_q,\bU_k)}\|_{2,F}\leq\|\bPhi_{k,\perp}\bPhi_k^{\dagger}\bU_k\|_{2,F}.
\end{eqnarray*}

\paragraph{Extracting the polynomials}
The expressions for $\bPhi_{k,\perp}$ in (\ref{e_phidecomp}) and
$\bPhi_k^{\dagger}$ in (\ref{e_proj}), and submultiplicativity (Section~\ref{sxn:la}) yield
\begin{eqnarray*}
\|\bPhi_{k,\perp}\bPhi_k^{\dagger}\bU_k\|_{2,F}&=&
\|\bU_{k,\perp}\,\phi(\bSigma_{k,\perp})\,\bV_{k,\perp}^T\bX(\bV_{k}^T\bX)^{\dagger}\,\phi(\bSigma_k)^{-1}\bU_k^T\bU_k\|_{2,F}\\
&=& \|\phi(\bSigma_{k,\perp})\,\bV_{k,\perp}^T\bX(\bV_{k}^T\bX)^{\dagger}\,\phi(\bSigma_k)^{-1}\|_{2,F}\\
&\leq &\|\phi(\bSigma_{k,\perp})\|_2\,\|\phi(\bSigma_k)^{-1}\|_2\>\|\bV_{k,\perp}^T\bX(\bV_{k}^T\bX)^{\dagger}\|_{2,F}.
\end{eqnarray*}
Combining the previous two sets of inequalities gives
\begin{eqnarray*}
\|\sin{\bTheta(\mathcal{K}_q,\bU_k)}\|_{2,F}\leq
\|\phi(\bSigma_{k,\perp})\|_2\,\|\phi(\bSigma_k)^{-1}\|_2\>\|\bV_{k,\perp}^T\bX(\bV_{k}^T\bX)^{\dagger}\|_{2,F}.
\end{eqnarray*}
This concludes the proof for general $\bX$. The proof
for the special case where $\bX$ has
linearly independent columns follows from Lemma~\ref{l_2}.

\section{Proof of Theorem~\ref{t_3}}\label{s_prooft3}
The proof imitates that of Theorem~\ref{t_2},
and simply substitutes the vectors $\bu_i$ for the matrix $\bU_k$.
Note that $(\bI-\bU_k \bU_k^T)\bu_i=0$ for $1\leq i \leq k$, which implies
\begin{eqnarray*}
|\sin{\bTheta(\mathcal{K}_q,\bu_i)}|\leq \|\bPhi_{k,\perp}\bPhi_k^{\dagger}\bu_i\|_{2},
\qquad 1\leq i\leq k.
\end{eqnarray*}
The expressions for $\bPhi_{k,\perp}$ in (\ref{e_phidecomp}) and
$\bPhi_k^{\dagger}$ in (\ref{e_proj}), and submultiplicativity yield
\begin{eqnarray*}
\|\bPhi_{k,\perp}\bPhi_k^{\dagger}\bu_k\|_{2}&=&
\|\bU_{k,\perp}\,\phi(\bSigma_{k,\perp})\,\bV_{k,\perp}^T\bX(\bV_{k}^T\bX)^{\dagger}\,\phi(\bSigma_k)^{-1}\bU_k^T\bu_i\|_{2}\\
&=& \|\phi(\bSigma_{k,\perp})\,\bV_{k,\perp}^T\bX(\bV_{k}^T\bX)^{\dagger}\,\phi(\sigma_i)^{-1}\|_2\\
&\leq &\|\phi(\bSigma_{k,\perp})\|_2\,|\phi(\sigma_i)^{-1}|\>\|\bV_{k,\perp}^T\bX(\bV_{k}^T\bX)^{\dagger}\|_{2}.
\end{eqnarray*}
Combining the previous two sets of inequalities gives
\begin{eqnarray*}
|\sin{\bTheta(\mathcal{K}_q,\bu_i)}|\leq
\|\phi(\bSigma_{k,\perp})\|_2\,|\phi(\sigma_i)^{-1}|\>\|\bV_{k,\perp}^T\bX(\bV_{k}^T\bX)^{\dagger}\|_{2}.
\end{eqnarray*}
This concludes the proof of the case for general $\bX$.
The proof for the special case where $\bX$ has orthonormal columns
follows from Lemma~\ref{l_2}.

\section{Proof of Theorem~\ref{t_4}}\label{s_prooft4}
This proof is more involved than the previous ones,
and requires two auxiliary results, an alternative
expression for the error (Section~\ref{s_aux1}), and
a bound on its Frobenius norm (Section~\ref{s_aux2}).

\subsection{An alternative expression for the error}\label{s_aux1}
Algorithm~\ref{alg:BL} approximates the dominant left singular vectors of $\bA$
by the orthonormal matrix $\hat\bU_k \in \real^{m \times k}$.
Since bounding $\|\bA-\hat\bU_k\hat\bU_k^T\bA\|_F$ seems hard, we
present an alternative expression that is easier to analyze.

\begin{lemma}[Lemma 8 in \cite{BDM2014}]\label{lem:restate}
Let $\bU_K$ be an orthonormal basis for $\mathcal{K}_q$ and let $\hat{\bU}_i$ be as in (\ref{e_hatui}), containing the top $i$ columns of the output of Algorithm~\ref{alg:BL}. Then
\begin{equation}
\bA-\hat\bU_i\hat\bU_i^T\bA = \bA-\bU_K \left(\bU_K^T\bA\right)_i,
\qquad 1\leq i\leq k.
\end{equation}
In addition,  $\bU_K \left(\bU_K^T\bA\right)_i$ is a best rank-$i$ approximation to
$\bA$ from $\mathcal{K}_q$ in the Frobenius norm,
\begin{equation}\label{eqn:opt1}
\FNormS{\bA-\bU_K \left(\bU_K^T\bA\right)_i} =
\min_{\rank(\bY)\leq i}\FNormS{\bA - \bU_K \bY}, \qquad 1\leq i \leq k.
\end{equation}
\end{lemma}

\begin{proof}
Since the transition to best rank-$i$ approximations is a key component, we illustrate
how it comes about by proving the first assertion for the case $i=k$.

Algorithm~\ref{alg:BL} outputs $\hat{\bU}_k=\bU_K\bU_{W,k}$, where
$\bU_{W,k}$ is the matrix of the dominant $k$ left singular vectors of
$\bW = \bU_K^T\bA$. This means $\bU_{W,k}$ spans the same range as $\bW_k$,
the best rank-$k$ approximation to $\bW$. Therefore
\begin{eqnarray*}
\bA-\hat\bU_k\hat\bU_k^T\bA &=& \bA-\bU_K \bU_{W,k}\bU_{W,k}^T\bU_K^T\bA\\
 &=& \bA-\bU_K \bW_k \bW_k^{\dagger} \bW = \bA-\bU_K \bW_k.
\end{eqnarray*}
The last equality follows from $\bW_k\bW_k^{\dagger}$ being the orthogonal projector onto
$\range(\bW_k)$.
\end{proof}

Lemma~\ref{lem:restate} shows that (\ref{eqn:th1}) in Theorem~\ref{t_4}
can be proved by bounding $\|\bA-\bU_K \left(\bU_K^T\bA\right)_i\|_F$.
Next we transition from the best rank-$i$ approximation of the "projected" matrix $(\bU_K^T\bA)_i$
to the best rank-$i$ approximation $\bA_i$ of the original matrix, by splitting for $1\leq i\leq k$,
\begin{eqnarray}\label{e_Ai}
\bA=\bA_i+ \bA_{i,\perp}\qquad \text{where} \quad\bA_i=\bU_i\bSigma_i\bV_i^T \  \text{and}\
\bA_{i,\perp}=\bU_{i,\perp}\bSigma_{i,\perp}\bV_{i,\perp}^T.
\end{eqnarray}

\begin{lemma}\label{l_aux2}
Let $\bU_K$ be an orthonormal basis for $\mathcal{K}_q$, and $\hat{\bU}_i$ in (\ref{e_hatui})
the columns of the output of Algorithm~\ref{alg:BL}. Then
$$\FNormS{\bA-\hat\bU_i\hat\bU_i^T\bA} \leq
 \FNormS{\bA_i-\bU_K \bU_K^T\bA_i} + \FNormS{\bA_{i,\perp}}.$$
\end{lemma}
\begin{proof}
The optimality of (\ref{eqn:opt1}) in Lemma~\ref{lem:restate} implies
\begin{eqnarray*}
\FNormS{\bA-\hat\bU_i\hat\bU_i^T\bA} &=&
\FNormS{\bA-\bU_K \left(\bU_K^T\bA\right)_i}\nonumber\\
&\leq& \FNormS{\bA-\bU_K \bU_K^T\bA_i} \nonumber\\
&=& \FNormS{\bA_i-\bU_K \bU_K^T\bA_i} + \FNormS{\bA_{i,\perp}}.
\end{eqnarray*}
The last equality follows from Lemma~\ref{l_pyth}.
\end{proof}

\subsection{Bounding the important part of the error}\label{s_aux2}
We bound the term in Lemma~\ref{l_aux2} over which we have control,
namely  $\FNormS{\bA_i-\bU_K \bU_K^T\bA_i}$.

As in Section~\ref{sxn:prooft2}, let $\mathcal{P}_q$ be the orthogonal projector onto $\mathcal{K}_q$.
For $\bPhi$ in (\ref{eqn:Phidef}) let
$\bPhi\bPhi^{\dagger}$ be the orthogonal projector onto $\range(\bPhi)$,
with $\range(\bPhi\bPhi^{\dagger})\subset\range(\mathcal{P}_q)$ due to (\ref{eqn:rangeK}).
The leads to the obvious bound
\begin{equation}\label{eqn:connectwithcheb}
\FNorm{\bA_i-\bU_K \bU_K^T\bA_i} = \FNorm{\bA_i-\mathcal{P}_q\bA_i} \leq \FNorm{\bA_i-\bPhi\bPhi^{\dagger} \bA_i}, \qquad 1\leq i\leq k.
\end{equation}
We don't stop here, though, but go further and pursue a bound in terms  of polynomials.
\begin{lemma}\label{lem:pd1}
Let $\phi(x)$ be a polynomial of degree $2q+1$ with odd powers only that satisfies
$\phi(\sigma_j)\geq \sigma_j$ for $1\leq j\leq  k$. Then
\begin{eqnarray*}\|\bA_i-\bU_K \bU_K^T\bA_i\|_F
\leq \|\bU_i\phi\left(\bSigma_i\right)-\bPhi \bPhi^{\dagger}\bU_i\phi\left(\bSigma_i\right)\|_F,
\qquad 1\leq i\leq k.
\end{eqnarray*}
\end{lemma}

\begin{proof}
We use the abbreviation $\mathcal{P}_{\phi}^{\perp}\equiv \bI-\bPhi\bPhi^{\dagger}$,
to denote the orthogonal projector onto $\range(\bPhi)^{\perp}$.
From (\ref{e_Ai}), (\ref{eqn:connectwithcheb})
and the unitary invariance of the Frobenius norm follows
\begin{eqnarray*}\label{eqn:d1}
\|\bA_i-\bU_K \bU_K^T\bA_i\|_F &\leq &
\|\mathcal{P}_{\phi}^{\perp}\bA_i\|_F= \|\mathcal{P}_q^{\perp}\,\bU_i\bSigma_i\|_F,\qquad
1\leq i\leq k.
\end{eqnarray*}
Expressing the squared Frobenius norm as a sum of squared column norms,
and then applying the assumption $\sigma_j\leq \phi(\sigma_j)$ yields for $1\leq i\leq k$,
\begin{eqnarray*}
\FNormS{\mathcal{P}_{\phi}^{\perp}\bU_i\bSigma_i} &=&
\sum_{j=1}^i{\sigma_j^2\|\mathcal{P}_{\phi}^{\perp}\bu_j\|_2^2}\leq
\sum_{j=1}^i{\phi(\sigma_j)^2\|\mathcal{P}_{\phi}^{\perp}\bu_j\|_2^2}=
\FNormS{\mathcal{P}_{\phi}^{\perp}\bU_i\phi(\bSigma_i)}.
\end{eqnarray*}
\end{proof}
\subsection{From projections to least-squares residuals}
Now we are ready to apply the approach from Theorem~\ref{t_2} and view
the result of Lemma~\ref{lem:pd1} as a least squares residual.

\begin{lemma}\label{l_aux3}
Under the assumptions of Theorem~\ref{t_4},
$$\|\bU_i\phi\left(\bSigma_i\right)-\bPhi \bPhi^{\dagger}\bU_i\phi\left(\bSigma_i\right)\|_F\leq
\|\phi(\bSigma_{k,\perp})\|_2\,\>\|\bV_{k,\perp}^T\bX(\bV_{k}^T\bX)^{\dagger}\|_{F}.$$
\end{lemma}

\begin{proof}
Based on the orthogonality $(\bI-\bU_k \bU_k^T)\bU_i=\bzero$ for $1\leq i\leq  k$,
we can deduce as in (\ref{e_ls1}) that
\begin{eqnarray*}
\FNorm{(\bI-\bPhi \bPhi^{\dagger})\,\bU_i\phi(\bSigma_i)}\leq
\|\bPhi_{k,\perp}\bPhi_k^{\dagger}\, \bU_i\phi(\bSigma_i)\|_{F}.
\end{eqnarray*}
The expressions for $\bPhi_{k,\perp}$ in~(\ref{e_phidecomp}) and  $\bPhi_k^{\dagger}$
in~(\ref{e_proj}), along with the strong submultiplicativity in Section~\ref{sxn:la} yield
\begin{eqnarray*}
\|\bPhi_{k,\perp}\bPhi_k^{\dagger}\bU_i\phi(\bSigma_i)\|_{F}&=&
\|\bU_{k,\perp}\,\phi(\bSigma_{k,\perp})\,\bV_{k,\perp}^T\bX(\bV_{k}^T\bX)^{\dagger}\,
\phi(\bSigma_k)^{-1}\bU_k^T\bU_i\phi(\bSigma_i)\|_{F}\\
&\leq &\|\phi(\bSigma_{k,\perp})\|_2\,\>\|\bV_{k,\perp}^T\bX(\bV_{k}^T\bX)^{\dagger}\|_{F}.
\end{eqnarray*}
The above inequality is obtained by noting that for $i=k$ we have
$\phi(\bSigma_k)^{-1}\bU_k^T\bU_i\phi\left(\bSigma_i\right)=\bI_k$, while for $1\leq i<k$,
$$\phi(\bSigma_k)^{-1}\bU_k^T\bU_i\phi\left(\bSigma_i\right)=
\begin{pmatrix} \bI_i\\ \bzero_{(k-i)\times i}\end{pmatrix}.$$

\end{proof}
\subsection{Concluding the proof of Theorem~\ref{t_4}}
We prove each of the three inequalities in turn. Recall that
$$\Delta \equiv \|\phi(\bSigma_{k,\perp})\|_2\>
\|\bV_{k,\perp}^T\bX\,(\bV_{k}^T\bX)^{\dagger}\|_{F}.$$

\subsubsection*{Proof of (\ref{eqn:th1})}
Combining Lemmas \ref{l_aux2}, \ref{lem:pd1}, and~\ref{l_aux3}
and recognizing the expression for $\Delta$ yields
\begin{eqnarray}\label{e_d3}
\|\bA-\hat{\bU}_i\hat{\bU}_i^T\|_F^2&\leq& \|\bA_{i,\perp}\|_F^2+
\|\phi(\bSigma_{k,\perp})\|_2^2\,\>\|\bV_{k,\perp}^T\bX(\bV_{k}^T\bX)^{\dagger}\|_{F}^2,
\nonumber\\
&=&\|\bA_{i,\perp}\|_F^2+ \Delta^2, \qquad 1\leq i\leq k.
\end{eqnarray}
Inserting $\|\bA_{i,\perp}\|_F = \|\bA-\bA_i\|_F$ gives
\begin{equation}\label{eqn:d3}
\FNormS{\bA-\hat\bU_i\hat\bU_i^T\bA} \leq \FNormS{\bA-\bA_{i}}+\Delta^2, \qquad
1\leq i\leq k.
\end{equation}
Taking advantage of the inequality below for scalars $\alpha, \beta\geq 0$,
\begin{eqnarray}\label{e_sqrt}
\sqrt{\alpha^2+\beta^2}\leq \sqrt{\alpha^2+\beta^2+2\alpha\beta}=
\sqrt{(\alpha+\beta)^2}=\alpha + \beta,
\end{eqnarray}
gives the weaker, but square-free bound (\ref{eqn:th1}).

\subsubsection*{Proof of (\ref{eqn:th2})}
We use \cite[Theorem 3.4]{G2014}, which shows that an additive error bound for a
low-rank approximation in the Frobenius norm implies the same in the two norm.

\begin{lemma}[Theorem 3.4 in \cite{G2014}]\label{l_gu}
Given  $\bA, \tilde{\bA} \in \real^{m \times n}$ with
$\rank(\tilde{\bA})=k<\rank(\bA)$. If
$\FNormS{\bA-\tilde{\bA}}\leq \FNormS{\bA-\bA_k}+\delta$, then
$$\TNormS{\bA-\tilde{\bA}}\leq \TNormS{\bA-\bA_k}+\delta.$$
\end{lemma}

Apply Lemma~\ref{l_gu} to (\ref{eqn:d3}), to get
$\TNormS{\bA-\hat\bU_{i}\hat\bU_{i}^T\bA} \leq \TNormS{\bA-\bA_i}+\Delta^2$,
and take square roots based on (\ref{e_sqrt}).

\subsubsection*{Proof of (\ref{eqn:th3})}
The upper bounds follow from the \textit{minimax theorem for singular values}
\cite[Theorem 8.6.1]{GovL13}.

This leaves the lower bounds. Recall the
non-increasing ordering of the singular values $\sigma_1\geq \cdots \geq \sigma_k$,
and the fact that $\hat{\bU}_i$ in (\ref{e_hatui}) has orthonormal columns.

\paragraph{Case $i=1$} Apply Lemma~\ref{l_pyth} to (\ref{e_d3})
\begin{eqnarray*}
\FNormS{\bA}-\FNormS{\hat{\bu}_{1}^T\bA} = \FNormS{\bA-\hat{\bu}_{1} \hat{\bu}_{1}^T\bA}
\leq \FNormS{\bA_{1,\perp}} + \Delta^2.
\end{eqnarray*}
From $\FNormS{\bA}-\FNormS{\bA_{1,\perp}} = \sigma_1^2$ follows
$\sigma_1^2\leq \FNormS{\bA-\hat{\bu}_{1} \hat{\bu}_{1}^T\bA} +\Delta^2$.
Taking square roots based on (\ref{e_sqrt}) proves (\ref{eqn:th3}) for $i=1$.

\paragraph{Case $2\leq i\leq k$}
Among all matrices of rank $i-1$, the matrix $\bA_{i-1}$ is closest to~$\bA$ in the Frobenius norm.
Hence
$$\|\bA_{i-1,\perp}\|_F=
\|\bA-\bA_{i-1}\|_F\leq \|\bA-\hat{\bU}_{i-1}\hat{\bU}_{i-1}^T\bA\|_F.$$
The above, together  with the outer product representation
$\hat{\bU}_{i} \hat{\bU}_{i}^T = \hat{\bU}_{i-1} \hat{\bU}_{i-1}^T + \hat{\bu}_{i}\hat{\bu}_{i}^T$,
Lemma~\ref{l_pyth} and (\ref{e_d3}) gives
\begin{eqnarray*}
\|\bA_{i-1,\perp}\|_F^2- \FNormS{\hat{\bu}_{i}\hat{\bu}_{i}^T\bA} &\leq&
\FNormS{\bA-\hat\bU_{i-1} \hat\bU_{i-1}^T\bA} - \FNormS{\hat{\bu}_{i}\hat{\bu}_{i}^T\bA}\\
&=& \FNormS{\bA-\hat\bU_{i} \hat\bU_{i}^T\bA}
\leq \FNormS{\bA_{i,\perp}} +\Delta^2.
\end{eqnarray*}
At last, applying $\FNormS{\bA_{i-1,\perp}}-\FNormS{\bA_{i,\perp}} = \sigma_{i}^2$, and taking
square roots based on (\ref{e_sqrt})
proves (\ref{eqn:th3}) for $2\leq i\leq k$.

This concludes the proof for general $\bX$. The proof for the special case where $\bX$ has
orthonormal columns follows from Lemma~\ref{l_2}.

\section{Conclusions and open problems}\label{s_conc}
Motivated by the emergence of randomized Krylov space methods
for low-rank approximations \cite{MM2015,Wang2015},
we presented a "proof of concept", that is, structural results for the accuracy of
approximate dominant subspaces.

Several open problems arise from our work:
\begin{enumerate}
\item Can we better understand and close the disconnect between low-rank approximations
and dominant subspace computations?

A singular value gap is a must for dominant subspace computations, if only to ensure wellposedness
of the mathematical problem. In contrast, low-rank approximations can do without a gap for
special starting guesses $\bX$ \cite{MM2015}. This comes at the detriment of accuracy, though.
Bolstered by a gap, subspace accuracy exhibits the logarithmic dependence (\ref{eqn:valq}) on~$\epsilon$,
while, without a gap, the accuracy of a low-rank approximation has only polynomial dependence on~$\epsilon$.
To the best of our knowledge, gap-independent results are not known for arbitrary~$\bX$.
As the analysis \cite{MM2015} only exploits the fact that $\bX$ can give an approximation
that is polynomially close to optimal in the Frobenius norm, it could potentially be extended
to a variety of random starting guesses.

\item Is it possible to relax  the full-rank assumption for $\bV_k^T\bX$?\\
Our proofs require $\rank(\bV_k^T\bX)=k$, which forces starting guesses
to have at least $s\geq k$ columns. Thus, even in the presence of the requisite singular value gaps,
our proofs collapse for starting guesses that consist of a single column.

\item Are our bounds tight enough to be informative, and how relevant are they for
practical numerical implementations of block Krylov methods?

\end{enumerate}
\section{Acknowledgments} We thank Mark Embree for many useful discussions.

\appendix
\section{More general proof of Theorem~\ref{t_2}}\label{s_prooft2a}
The proof below applies to starting guesses $\bX$ with linearly independent columns and
consists of several steps.

\paragraph{Preparing $\bX$}
Since subspace  angles are defined by matrices with orthonormal columns, we perform
a thin QR decomposition $\bX=\bQ\bR$,  where $\bQ\in\real^{n\times s}$ has orthonormal
columns, and $\bR\in\real^{s\times s}$ is nonsingular. Then $\range(\bQ)=\range(\bX)$.

The expression for $\bPhi$ contains a basis transformation on $\bX$ with the orthogonal
matrix~$\bV$, resulting in a $n\times s$ matrix
\begin{eqnarray*}
\bV^T\bQ=\begin{pmatrix}\bV_k^T\bQ\\ \bV_{k,\perp}^T\bQ\end{pmatrix}
=\begin{pmatrix}\bQ_k\\ \bQ_{k,\perp}\end{pmatrix}
\end{eqnarray*}
with orthonormal columns. It remains to account for  $\bR$:
\begin{eqnarray*}
\bV^T\bX=(\bV^T\bQ)\,\bR=\begin{pmatrix}\bQ_k^T\bR\\ \bQ_{k,\perp}^T\bR\end{pmatrix}
=\begin{pmatrix}\bX_k\\ \bX_{k,\perp}\end{pmatrix}.
\end{eqnarray*}
By assumption,
$k=\rank(\bV_k^T\bQ)=\rank(\bX_k)=\rank(\bQ_k)$, so that
$\bX_k\in\real^{k\times s}$ and $\bQ_k\in\real^{k\times S}$ have full row rank.
In particular,
\begin{eqnarray}\label{ea_Xk}
\bQ_k\bQ_k^{\dagger}=\bI_k, \qquad
\bQ_k^{\dagger}=\bQ_k^T(\bQ_k\bQ_k^T)^{-1}.
\end{eqnarray}
For $\bX_k$, though, we forego the Moore-Penrose inverse,
and choose instead a $(1,2,3)$ inverse \cite[Definition 6.2.4]{CaM79}.
The matrix
$$\bX_k^{+}\equiv \bR^{-1}\bQ_k^{\dagger}$$
is a right inverse, $\bX_k\bX_k^{+}=\bI_k$, and
satisfies three of the four Moore-Penrose conditions,
$$\bX_k\,\bX_k^{+}\,\bX_k=\bX_k, \qquad \bX_k^{+}\,\bX_k\,\bX_k^{+}=\bX_k^{+}, \qquad
(\bX_k\,\bX_k^{+})^T=(\bX_k\,\bX_k^{+}).$$

\paragraph{Viewing the sine as a least squares residual}
Let $\mathcal{P}_q$ be the orthogonal projector onto the Krylov space $\mathcal{K}_q$,
and $\bPhi\bPhi^{\dagger}$ the orthogonal projector onto $\range(\bPhi)$.
From $\range(\bPhi)\subset\mathcal{K}_q$
follows $\range(\bPhi\bPhi^{\dagger})\subset\range(\mathcal{P}_q)$, hence (\ref{e_dist}) implies
\begin{eqnarray}\label{ea_sin}
\|\sin{\bTheta(\mathcal{K}_q,\bU_k)}\|_{2,F}&=&\|(\bI-\mathcal{P}_q)\,\bU_k\|_{2,F}
\|(\bI-\bPhi\bPhi^{\dagger})\,\bU_k\|_{2,F}\\
&=& \|\sin{\bTheta(\bPhi,\bU_k)}\|_{2,F}.\nonumber
\end{eqnarray}
Lemma~\ref{lem:2norm} shows that
$\|(\bI-\bPhi\bPhi^{\dagger})\,\bU_k\|_{2,F}$ is the residual
of the least squares problem
$$\|(\bI-\bPhi\bPhi^{\dagger})\,\bU_k\|_{2,F}=\min_{\Psi}{\|\bU_k-\bPhi\bPsi\|_{2,F}}
=\|\bU_k-\bPhi\bPsi_{opt}\|_{2,F},$$
where $\bPsi_{opt}=\bPhi^{\dagger}\bU_k$ is a least squares solution.

\paragraph{Focussing on the target space}
Decompose $\bPhi$ into the target component $\range(\bU_k)$ and the complementary subspace,
$\bPhi=\bPhi_k+\bPhi_{k,\perp}$, where
\begin{eqnarray}\label{ea_phidecomp}
\bPhi_k\equiv\bU_k\phi(\bSigma_k)\bX_k, \qquad
\bPhi_{k,\perp}\equiv \bU_{k,\perp}\phi(\bSigma_{k,\perp})\bX_{k,\perp}.
\end{eqnarray}
It is easy to verify that
\begin{eqnarray}\label{ea_phikinv}
\bPhi_k^{+}\equiv \bX_k^{+} \phi(\bSigma_k)^{-1}\bU_k^T
\end{eqnarray}
satisfies the conditions of a $(1,2,3)$ inverse.
The minimality of the least squares residual implies
\begin{eqnarray*}\label{ea_ls}
\|(\bI-\bPhi\bPhi^{\dagger})\,\bU_k\|_{2,F}&=&\|\bU_k-\bPhi\,(\bPhi^{\dagger}\bU_k)\|_{2,F}\nonumber\\
&\leq&  \|\bU_k-\bPhi\,(\bPhi_k^{+}\bU_k)\|_{2,F}
=\|(\bI-\bPhi\bPhi_k^{+})\,\bU_k\|_{2,F}.
\end{eqnarray*}
Now replace the other instance of $\bPhi$ by (\ref{ea_phidecomp}),
\begin{eqnarray}\label{ea_ls1}
\|(\bI-\bPhi\bPhi^{+})\,\bU_k\|_{2,F}&\leq &\|(\bI-\bPhi\bPhi_k^{+})\,\bU_k\|_{2,F}\nonumber\\
&=&\|(\bI-\bPhi_k\bPhi_k^{+})\bU_k-\bPhi_{k,\perp}\bPhi_k^{+}\bU_k\|_{2,F}.
\end{eqnarray}
Since a $(1,2,3)$ inverse is an orthogonal projector \cite[(2.2.13) in Section 2.2.1]{Bjo15}, it follows that
$\bPhi_k\bPhi_k^{+}=\bU_k\bU_k^T$
is the orthogonal projector onto the target space $\range(\bU_k)$.
This observation simplifies (\ref{ea_ls1}),
\begin{eqnarray*}
\|(\bI-\bPhi_k\bPhi_k^{+})\bU_k-\bPhi_{k,\perp}\bPhi_k^{+}\bU_k\|_{2,F}
&=&\|(\bI-\bU_k\bU_k^{T})\bU_k-\bPhi_{k,\perp}\bPhi_k^{+}\bU_k\|_{2,F}\\
&=&\|\bPhi_{k,\perp}\bPhi_k^{+}\bU_k\|_{2,F}.
\end{eqnarray*}

\paragraph{Summary so far}
Combining the above with (\ref{ea_sin}) and (\ref{ea_ls1}) gives
\begin{eqnarray*}
\|\sin{\bTheta(\mathcal{K}_q,\bU_k)}\|_{2,F}\leq
\|\sin{\bTheta(\bPhi,\bU_k)}\|_{2,F}\leq \|\bPhi_{k,\perp}\bPhi_k^{+}\bU_k\|_{2,F}.
\end{eqnarray*}

\paragraph{Extracting the polynomials}
The expressions for $\bPhi_{k,\perp}$ in (\ref{ea_phidecomp}) and
$\bPhi_k^{+}$ in (\ref{ea_phikinv}), and submultiplicativity (Section~\ref{sxn:la}) yield
\begin{eqnarray*}
\|\bPhi_{k,\perp}\bPhi_k^{+}\bU_k\|_{2,F}&=&
\|\phi(\bSigma_{k,\perp})\,\bX_{k,\perp}\bX_k^{+}\,\phi(\bSigma_k)^{-1}\|_{2,F}\\
&\leq &\|\phi(\bSigma_{k,\perp})\|_2\,\|\phi(\bSigma_k)^{-1}\|_2\>\|\bX_{k,\perp}\bX_k^{+}\|_{2,F}.
\end{eqnarray*}
Combining the previous two sets of inequalities gives
\begin{eqnarray*}
\|\sin{\bTheta(\mathcal{K}_q,\bU_k)}\|_{2,F}\leq
\|\phi(\bSigma_{k,\perp})\|_2\,\|\phi(\bSigma_k)^{-1}\|_2\>\|\bX_{k,\perp}\bX_k^{+}\|_{2,F}.
\end{eqnarray*}
We chose the $(1,2,3)$ inverse so that $\bR$ cancels out,
$\|\bX_{k,\perp}\bX_k^{+}\|_{2,F}=\|\bQ_{k,\perp}\bQ_k^{\dagger}\|_{2,F}$, and
\begin{eqnarray*}
\|\sin{\bTheta(\mathcal{K}_q,\bU_k)}\|_{2,F}\leq
\|\phi(\bSigma_{k,\perp})\|_2\,\|\phi(\bSigma_k)^{-1}\|_2\>\|\bQ_{k,\perp}\bQ_k^{\dagger}\|_{2,F}.
\end{eqnarray*}
At last, Lemma~\ref{l_2} and $\range(\bQ)=\range(\bX)$ imply
$$\|\bQ_{k,\perp}\bQ_k^{\dagger}\|_{2,F}=\|\tan{\bTheta(\bQ,\bV_k)}\|_{2,F}
=\|\tan{\bTheta(\bX,\bV_k)}\|_{2,F}.$$

\bibliography{BlockKrylov}

\end{document}